\documentclass[11pt]{amsart}

\usepackage[colorlinks=true, pdfstartview=FitV, linkcolor=blue,
citecolor=blue]{hyperref}

\usepackage{a4wide,color}
\usepackage{amssymb,amsmath,amscd}

\DeclareFontFamily{U}{mathx}{\hyphenchar\font45}
\DeclareFontShape{U}{mathx}{m}{n}{ <5> <6> <7> <8> <9> <10>
   <10.95> <12> <14.4> <17.28> <20.74> <24.88> mathx10 }{}
\DeclareSymbolFont{mathx}{U}{mathx}{m}{n}
\DeclareFontSubstitution{U}{mathx}{m}{n}
\DeclareMathAccent{\widecheck}{0}{mathx}{"71}

\theoremstyle{plain}
\newtheorem{theorem}{Theorem}[section]
\newtheorem{prop}[theorem]{Proposition}
\newtheorem{lemma}[theorem]{Lemma}
\newtheorem{coro}[theorem]{Corollary}
\newtheorem{fact}[theorem]{Fact}

\theoremstyle{definition}
\newtheorem{definition}[theorem]{Definition}
\newtheorem{example}[theorem]{Example}
\newtheorem{remark}[theorem]{Remark}

\newcommand{\ts}{\hspace{0.5pt}}
\newcommand{\nts}{\hspace{-0.5pt}}

\newcommand{\CC}{\mathbb{C}\ts}
\newcommand{\RR}{\mathbb{R}\ts}
\newcommand{\QQ}{{\ts \mathbb{Q}}}
\newcommand{\NN}{\mathbb{N}}

\newcommand{\cS}{\mathcal{S}}

\newcommand{\vL}{\varLambda}
\newcommand{\ii}{\mathrm{i}\ts}
\newcommand{\ee}{\mathrm{e}}
\newcommand{\dd}{\, \mathrm{d}}

\newcommand{\exend}{\hfill $\Diamond$}
\newcommand{\defeq}{\mathrel{\mathop:}=}

\newcommand{\fn}[1]{{\left\vert\kern-0.25ex\left\vert\kern-0.25ex\left\vert #1
    \right\vert\kern-0.25ex\right\vert\kern-0.25ex\right\vert_1}}

\DeclareMathOperator{\card}{card}

\DeclareMathOperator{\supp}{supp}
\DeclareMathOperator{\real}{Re}
\DeclareMathOperator{\imag}{Im}
\DeclareMathOperator{\sinc}{sinc}

\DeclareMathOperator*{\Conv}{\raisebox{-4pt}{\mbox{\Huge $\ast$}}}

\newcommand{\Cc}{C_{\mathsf{c}}}

\newcommand{\Cu}{C_{\mathsf{u}}}

\newcommand{\myfrac}[2]{\frac{\raisebox{-2pt}{$#1$}}
      {\raisebox{0.5pt}{$#2$}}}

\begin{document}

\title{A note on tempered measures}

\author{Michael Baake}
\address{Fakult\"at f\"ur Mathematik,
         Universit\"at Bielefeld, \newline
\hspace*{\parindent}Postfach 100131, 33501 Bielefeld, Germany}
\email{mbaake@math.uni-bielefeld.de }

\author{Nicolae Strungaru}
\address{Department of Mathematical Sciences,
         MacEwan University, \newline
\hspace*{\parindent}10700 \ts 104 Avenue,
         Edmonton, AB, Canada T5J 4S2}
\email{strungarun@macewan.ca}



\begin{abstract}
  The relation between tempered distributions and measures is analysed
  and clarified. While this is straightforward for positive measures,
  it is surprisingly subtle for signed or complex measures.
\end{abstract}

\maketitle

\section{Introduction}

Tempered distributions are the objects of choice for many problems in
harmonic analysis on $\RR^d$, with manifold applications for instance
in mathematical physics. Sometimes, however, one has to deal with
unbounded Radon measures in full generality. While the relations
between them are fairly straightforward for \emph{positive} tempered
measures, things become more subtle for signed or complex measures.

Though this complication is well known in principle \cite{ARMA1}, it
is a bit hidden in the literature and continues to lead to some
typical mistakes.  This is our motivation for this little note, which
is meant to provide the general connection in full generality, stated
as explicitly and concretely as possible.

In fact, we begin with the general case in Section~\ref{sec:general},
where we treat the positive and the general measures separately.  For
the critical statement that a signed or complex tempered measure need
not be slowly increasing, we provide constructive counterexamples in
Section~\ref{sec:counter}. Finally, in Section~\ref{sec:finite}, we
consider the special situation of measures with uniformly discrete
support, which is of particular relevance in the spectral theory of
aperiodic order \cite{TAO1,TAO2}.

\bigskip

\section{The general case}\label{sec:general}

Throughout, $\cS (\RR^d)$ denotes the space of Schwartz functions on
$\RR^d$ and $\cS' (\RR^d)$ the space of tempered distributions, all in
the sense of \cite{Schw}. Clearly, $\cS (\RR^d)$ contains
$C^{\infty}_{\mathsf{c}} (\RR^d)$, the space of $C^{\infty}$-functions
with compact support. For general background and results on Radon
measures, we refer to \cite{Edwards}. If $\mu$ is a positive
measure on $\RR^d$, we write $L^1 (\mu)$ for $L^1(\RR^d, \mu)$.

\begin{definition}\label{def:temp}
  Let $\mu$ be a Radon measure on $\RR^d$. It is a \emph{tempered
    measure} if there exists some $T\in \cS' (\RR^d)$ such that
  $\mu (\varphi) = T (\varphi)$ holds for all
  $\varphi\in C^{\infty}_{\mathsf{c}} (\RR^d)$.

  Further, $\mu$ is called \emph{strongly tempered} when, for all
  $\psi\in\cS(\RR^d)$, we have
  $\lvert \psi \rvert \in L^{1} \bigl( \lvert \mu \rvert \bigr)$
  together with the property that
  $\psi \mapsto \int_{\RR^d} \psi (x) \dd \mu (x) $ defines a tempered
  distribution.
\end{definition}

Here, the first part is the definition of \cite{ARMA1,Schw}, while the
second essentially is the definition from \cite{Kab}, though some care
has to be exercised when it comes to general Radon measures in
comparison to positive ones.

\begin{definition}\label{def:slow}
  A Radon measure $\mu$ on $\RR^d$ is called \emph{slowly increasing}
  if
\[
     \int_{\RR^d}  \frac{\dd | \mu | (x)}{ 1 +
     \lvert P (x) \rvert} \, < \, \infty
\]
holds for some polynomial
$P \in \CC \bigl[ x^{}_{1}, \ldots , x^{}_{d} \bigr]$.
\end{definition}

The second notion in Definition~\ref{def:temp} was originally
introduced in \cite{SpiStru} in a different way, by saying that a
measure is strongly tempered when it is slowly increasing. We shall
later show that these two definitions are equivalent.

Let us begin with a straightforward consequence of our definitions.

\begin{lemma}\label{lem:implications}
  Let\/ $\mu$ be a Radon measure on\/ $\RR^d$. If\/ $\mu$ is slowly
  increasing, it is also strongly tempered. Any strongly tempered
  measure is also tempered.
\end{lemma}

\begin{proof}
   The first claim follows from the observation that
\[
\begin{split}
    \big\lvert \mu (\psi) \big\rvert \, & \leqslant
    \int_{\RR^d} \lvert \psi (x) \rvert \dd \lvert \mu \rvert (x)
    \, = \int_{\RR^d} \big\lvert \psi (x) \big\rvert \bigl( 1 +
    \lvert P (x) \rvert \bigr) \, \frac{\dd \lvert \mu \rvert (x)}
    {1 + \lvert P (x) \rvert} \\[2mm] & \leqslant \,
    \big\| \bigl( 1 + \lvert P \rvert \bigr) \ts \psi \big\|_{\infty}
    \int_{\RR^d} \frac{\dd \lvert \mu \rvert (x)}
    {1 + \lvert P (x) \rvert}
\end{split}
\]
holds for any $\psi \in \cS (\RR^d)$, where the last integral is a
finite constant.

The second claim is obvious. Indeed, if $\mu$ is strongly
tempered, $ \psi \mapsto
T (\psi) = \int_{\RR^d} \psi (x) \dd \mu (x) $
defines a tempered distribution. Moreover, for any
$\varphi \in C^{\infty}_{\mathsf{c}} (\RR^d)$, we have
$T (\varphi) = \mu (\varphi) $ by definition.
\end{proof}

Let us continue with one result that looks technical, but is actually
fundamental. Its proof follows the arguments from
\cite[Thm.~2.1]{Kab}, but we present it in full detail here for
improved readability and because various of our deviations will become
significant later. To increase its readability, we split it into a
preliminary lemma and the main result as Proposition~\ref{prop:L-one}.

For simplicity, for multi-indices $\alpha$ and $\beta$,
 we denote by $\| . \|^{}_{\alpha, \beta}$ the Schwartz norm,
\[
  \| f \|^{}_{\alpha, \beta} \, \defeq  \sup_{x\in\RR^d}
  \big\lvert x^{\beta} \ts D^{\alpha} \nts f (x) \big\rvert \ts ,
\]
with $x^{\beta} = x^{\beta_1}_{1} \nts \cdots \ts x^{\beta_d}_{d}$
and $D^{\alpha} = \bigl( \frac{\partial}{\partial x^{}_{1}}
\bigr)^{\alpha^{}_1} \cdots \bigl( \frac{\partial}
{\partial x^{}_{d}}\bigr)^{\alpha^{}_d}$ as usual.

\begin{lemma}\label{lem:ex-pos}
  Let\/ $k_n \in \NN$ and\/ $c_n \in (0, \infty)$ define two sequences
  with the following properties,
\begin{enumerate}\itemsep=2pt
  \item $k_1 \geqslant 4 \ts$ and\/ $k_{n+1} \geqslant k_n +4$ for
      all\/ $n \geqslant 1\ts $,
  \item for all\/ $N\in\NN$, the sequence\/
      $\bigl( c^{}_n \ts 2^{(k_n - 3) N}\bigr)_{n\in\NN}$ is bounded.
\end{enumerate}
Then, there exists some non-negative\/ $\psi \in \cS(\RR^d)$ such that
$ \psi(x) = c_n$ holds for all\/ $n\in\NN$ and all\/ $x$ with\/
$2^{\ts k_n-1} \leqslant |x|^{}_{2} \leqslant 2^{\ts k_n+1}$, where\/
$|.|^{}_{2}$ refers to the Euclidean norm on\/ $\RR^d$.
\end{lemma}

\begin{proof}
  Select a non-negative function
  $\varphi\in C^{\infty}_{\mathsf{c}} (\RR^d)$ with $\varphi (x) = 1$
  for all $4 \leqslant \lvert x \rvert^{}_{2} \leqslant 16$ and with
  $\supp (\varphi) \subset \{ x : 2 < \lvert x \rvert^{}_{2} < 32 \}$,
  and set $\varphi^{}_{n} (x) = \varphi \bigl(x/2^{k_n - 3} \bigr)$
  for each $n$, which all are non-negative functions. Also, one has
  $\varphi^{}_{n} (x) = 1$ for all
  $2^{k_n-1} \leqslant |x|^{}_{2} \leqslant 2^{k_n+1}$ together with
\[
    \supp (\varphi^{}_{n}) \, \subset \, \big\{ x : 2^{k_n - 2} <
    \lvert x \rvert^{}_{2} < 2^{k_n + 2} \bigr\}.
\]
In particular, the functions $\varphi^{}_n$ have pairwise disjoint
supports. Next, consider the non-negative function
$\psi \defeq \sum_{n=1}^{\infty} c^{}_n \ts \varphi^{}_{n}$ which
satisfies the properties guaranteed by Lemma~\ref{lem:ex-pos}.
Therefore, if we show that $\psi \in \cS(\RR^d)$, we are done.

Let $\alpha$ and $\beta$ be arbitrary multi-indices, and set
$N=|\beta|-|\alpha|$, where
$|\beta| = \beta^{}_1 + \ldots + \beta^{}_d$ as usual. By (2), there
exist constants $C_{\alpha, \beta} = C_{\alpha, \beta} (N)$ such that
\[
  c^{}_n \ts 2^{(k_n - 3) N}
  \, \leqslant \, C^{}_{\alpha, \beta}
  \qquad \text{holds for all } \ts n \in \NN \ts .
\]
For arbitrary but fixed $x \in \RR^d$, one of the following
two cases applies.

\emph{Case $1$}. There is no $n \in \NN$ such that
$ 2^{k_n - 2} \leqslant \lvert x \rvert^{}_{2} \leqslant 2^{k_n + 2}$.
Then, we have $\psi \equiv 0$ in a neighbourhood of $x$ by
construction, and hence
\[
  \big| \ts x^{\beta} D^{\alpha} \psi (x)
  \big| \, =  \, 0 \ts .
\]

\emph{Case $2$}. There is some $n \in \NN$ such that
$ 2^{k_n - 2} \leqslant \lvert x \rvert^{}_{2} \leqslant 2^{k_n + 2}$.
Then, since $k_{j+1} > k_{j}+4$ for all $j$, this $n$ is unique.  The
pairwise disjoint supports of the functions $\varphi^{}_{n}$ then
imply that $\psi  =  c^{}_n \varphi^{}_n $
holds in a neighbourhood of $x$. Therefore, we get
\[
  \big| x^{\beta}  D^{\alpha} \psi (x) \big| \,
   =  \, \big| x^{\beta} D^{\alpha} c^{}_n \varphi^{}_n (x) \big|
  \, =  \, c^{}_n  \big| x^{\beta} D^{\alpha} \bigl(
    \varphi \bigl(x/2^{k_n - 3} \bigr) \bigr) \big|
  =  \, \frac{ c^{}_n  \lvert x^{\beta} \rvert}
   {2^{(k_n - 3)\ts |\alpha|}} \big| \bigl(  D^{\alpha}\varphi \bigr)
   \bigl(x/2^{k_n - 3} \bigr) \big|  .
\]
With $x = 2^{k_n - 3} \ts y$, this gives
\begin{align*}
   \big| x^{\beta}  D^{\alpha} \psi (x) \big| \, & =
    \, \frac{ c^{}_n   \ts 2^{(k_n - 3) \ts \lvert \beta \rvert}
    \ts \lvert y^{\beta} \rvert}{ 2^{(k_n - 3)\ts |\alpha|}} \big|
    \bigl( D^{\alpha}\varphi \bigr)  (y) \big|  \, =  \,
    c^{}_n \ts 2^{(k_n - 3) \ts  (|\beta| -|\alpha| )}  \big| y^{\beta}
     \bigl( D^{\alpha} \varphi \bigr)   (y) \big| \\[2mm]
    &\leqslant \, c^{}_n \ts 2^{(k_n - 3) \ts (|\beta|-|\alpha|)}
      \| \varphi \|^{}_{\alpha, \beta} \,  \leqslant \,
      C^{}_{\alpha, \beta}   \| \varphi \|^{}_{\alpha, \beta} \ts .
\end{align*}

For any $x \in \RR^d$, one of the two cases applies, and we thus
always obtain the estimate
$ \big| x^{\beta} D^{\alpha} \psi (x) \big| \leqslant C^{}_{\alpha,
  \beta} \| \varphi \|^{}_{\alpha, \beta} $ and hence also
\[
   \| \psi \|^{}_{\alpha, \beta} \, \leqslant \,
   C^{}_{\alpha, \beta}   \| \varphi \|^{}_{\alpha, \beta} \ts .
\]
Since the multi-indices $\alpha$ and $\beta$ were arbitrary, this
estimate shows that $\psi \in \cS(\RR^d)$, thus completing the proof.
\end{proof}

As a consequence, we obtain the following result, which is the key to
relating our notions to the class of \emph{positive} Radon measures.

\begin{prop}\label{prop:L-one}
  Let\/ $\mu$ be a positive Radon measure on\/ $\RR^d$ such that all\/
  $\varphi\in\cS (\RR^d)$ with\/ $\varphi \geqslant 0$ satisfy\/
  $\varphi\in L^{1} (\mu)$. Then, $\mu$ is slowly increasing.
\end{prop}

\begin{proof}
Set $A^{}_{0} = \{ x\in\RR^d : \lvert x \rvert^{}_{2} \leqslant 1 \}$,
and $A_{j} = \{ x \in \RR^d : 2^{j-1} \leqslant \lvert x \rvert^{}_{2}
\leqslant 2^{j} \}$ for $j\in\NN$. We now show that there are
constants $c>0$ and $a \geqslant 0$ such that
\begin{equation}\label{eq:eq-1}
  \mu (A_j) \, \leqslant \, c \, 2^{a j} \quad
  \text{holds for all } j \geqslant 0 \ts .
\end{equation}
Assume this to be false. Then, for all $c>0$ and $a \geqslant 0$,
there is some $j = j(c,a)$ with $\mu (A_j) > c \, 2^{a j}$.
Setting $c=1$, we get, for all $\ell\in\NN_0$, some $k^{}_{\ell} \in
\NN_0$ such that
\begin{equation}\label{eq:eq-2}
   \mu (A^{}_{k_{\ell}}) \, > \, 2^{\ts \ell \ts k_{\ell}} .
\end{equation}
In fact, for each $\ell\in\NN_0$, there must be infinitely many
such $k^{}_{\ell}$. Indeed, assume to the contrary that there is
some $\ell^{}_{0}$ for which
\begin{equation}\label{eq:eq-3}
   \mu (A_k) \, > \, 2^{\ts\ell_0 \ts k}
\end{equation}
holds only for finitely many $k$, say for $k^{}_{1}, \ldots ,
k_j$. Then, for each $1\leqslant i \leqslant j$, we can find some
$\ell_i$ with $\mu (A_i) < 2^{\ts \ell_i \ts k_i}$. Consequently, for
$\ell > \max \{ \ell^{}_{0}, \ell^{}_{1}, \ldots, \ell^{}_{j} \}$, one
has
\[
  \mu (A_{k_i}) \, < \, 2^{\ts \ell_i \ts k_i}
  \, < \, 2^{\ts \ell \ts k_i}
\]
for $1\leqslant i \leqslant j$ together with
$ \mu (A_{k}) < 2^{\ts \ell_{0} \ts k}
    \, < \, 2^{\ts \ell \ts k} $
for all $k \not\in \{ k^{}_{1}, \ldots , k^{}_{j} \}$.
This shows that \eqref{eq:eq-2} holds for infinitely
many integers $k$.

Consequently, for all $\ell\in\NN_0$, there are infinitely many $k$
with $\mu (A_k) > 2^{k \ts \ell}$. We can then construct a sequence
$4 < k^{}_{1} < k^{}_{2} < \cdots$ such that
$k^{}_{j+1} > k^{}_{j} + 4$ holds for all $j\in\NN$ together with
$\mu (A_{k_j}) > 2^{j \ts k_j}$.

Setting $c^{}_{j} = 1/\mu(A_{k_j})$, we see that
$c^{}_{j} < 2^{- j  k_j}$ and hence, for all $N \in \NN$, we have
 \[
    \limsup_{n \to\infty}  \, c^{}_n \ts 2^{(k_n - \ts 3) N}
    \, \leqslant \, \limsup_{n\to\infty}  2^{(k_n - \ts 3)
    N\nts - \ts n k_n} \, = \, 0 \ts .
\]
In particular, $c^{}_n \ts 2^{(k_n - 3) \ts N}$ is bounded for all
$N \in \NN$. It follows that $k^{}_{n}$ and $c^{}_{n}$ satisfy the
conditions of Lemma~\ref{lem:ex-pos}. Consequently, there exists some
non-negative function $\psi \in \cS(\RR^d)$ with
$ \psi(x) = 1/\mu(A_{k_n})$ for all $ x \in A^{}_{k_n}$.  Since the
sets $A_{k_n}$ are pairwise disjoint, it follows that, for any
$N \in \NN$, we have
\[
  \mu(\psi) \,
   \geqslant \int^{}_{\bigcup_{n=1}^{N} A^{}_{k_n} } \psi(x) \dd \mu(x)
    \, = \ts \sum_{n=1}^{N}  \int^{}_{A_{k_n}} \psi(x) \dd \mu(x)
   = \ts \sum_{n=1}^{N} \int^{}_{A^{}_{k_n}} \frac{\dd \mu (x)}
   {\mu(A^{}_{k_n})}  \, =  \, N  ,
\]
which contradicts the fact that $\psi \in L^1(\mu)$. So, our
assumption is wrong. This shows that there are some constants $c>0$
and $a \geqslant 0$ such that \eqref{eq:eq-1} holds for all
$j \geqslant 0$.  Also, after possibly replacing $a$ by a larger
number, we may assume $a\in\NN$ without loss of generality.

Then, since we have $\RR^d =\bigcup_{j=0}^{\infty} A_j$, where the
$A_j$ have disjoint interior but some common boundary for consecutive
values of $j$, we also have
\[
  \int^{}_{\RR^d} \frac{\dd \mu (x)}{1+|x|^{a+1}_{2}}
  \, \leqslant  \int^{}_{A^{}_0} \frac{\dd \mu (x)}{1+|x|^{a+1}_{2}}
       \: + \ts  \sum_{j=1}^{\infty} \int^{}_{A_j}
       \frac{\dd \mu (x)}{1+|x|^{a+1}_{2}} \ts .
\]
Since $A^{}_0$ is compact, we clearly have
\[
  I^{}_0 \, \defeq \int^{}_{A^{}_0}
  \frac{\dd \mu (x)}{1+|x|^{a+1}_{2}}
  \, \leqslant \, \mu ( A^{}_{0}) \, < \, \infty \ts .
\]
Next, for all $j \geqslant 1$, we get
\begin{align*}
  I_j \, & \defeq \int^{}_{A_j} \frac{\dd \mu (x)}{1+|x|^{a+1}_{2}}
           \,=  \int^{}_{2^{j-1} \leqslant \lvert x \rvert^{}_{2} \leqslant 2^{j+1}}
           \frac{\dd \mu (x)}{1+|x|^{a+1}_{2}} \\[2mm]
     \, & \leqslant  \frac{\mu (A_j )}{1+2^{(j-1)(a+1)}}
       \, \leqslant \, \frac{c \, 2^{a j}}{2^{(j-1) (a+1)}}
       \, = \, c \, 2^{- j} \ts  2^{\ts a + 1}  .
\end{align*}
This shows that
\[
  \int^{}_{\RR^d} \frac{\dd \mu (x) }{1+|x|^{a+1}_{2}}
  \, \leqslant \, I^{}_{0} \: + \,  c \, 2^{\ts a + 1} \ts
  \sum_{j=1}^{\infty} 2^{-j} \, = \, I^{}_{0} \: + \,
  c \, 2^{\ts a + 1} \, < \, \infty \ts ,
\]
which completes the proof.
\end{proof}

At this point, we get the following result.

\begin{theorem}\label{thm:positive}
  For a positive Radon measure\/ $\mu$ on\/ $\RR^d$, the following
  properties are equivalent:
\begin{enumerate}\itemsep=2pt
\item $\mu$ is slowly increasing;
\item $\mu$ is strongly tempered;
\item one has\/ $\lvert \psi \rvert \in L^{1} (\mu)$
    for all\/ $\psi \in \cS (\RR^d)$;
\item one has\/ $\psi \in L^{1} (\mu)$ for all\/
    $\psi \in \cS (\RR^d)$ with\/ $\psi \geqslant 0$;
\item $\mu$ is tempered.
\end{enumerate}
\end{theorem}

\begin{proof}
  $(1) \Rightarrow (2)$ is the first claim of
  Lemma~\ref{lem:implications}, while $(2) \Rightarrow (3)$ follows
  from Definition~\ref{def:temp} and $(3) \Rightarrow (4)$ is obvious.
  Further, $(4) \Rightarrow (1)$ is Proposition~\ref{prop:L-one},
  while $(2) \Rightarrow (5)$ is the second claim of
  Lemma~\ref{lem:implications}.

  To complete the proof, we could infer \cite[p.~242]{Schw} to obtain
  $(5) \Rightarrow (1)$, but we prefer to show $(5) \Rightarrow (4)$
  as follows. Let $\psi$ be any fixed, non-negative Schwartz function.
  Then, invoking a minor variant of the $C^{\infty}$ partitions of
  unity \cite[p.~299]{Lang}, there is a sequence of functions
  $\varphi^{}_{n} \in C^{\infty}_{\mathsf{c}} (\RR^d)$ with
  $\varphi^{}_{1} \leqslant \varphi^{}_{2} \leqslant \ldots$ such that
  $\varphi^{}_{n} = 1$ on $\{ \lvert x \rvert \leqslant n\}$ and
  $\varphi^{}_{n} \ts \psi \xrightarrow{n\to\infty} \psi$ in
  $\cS (\RR^d)$.

  Since $\mu$ is tempered, there is some $T \in \cS' (\RR^d)$ such
  that $T (\varphi) = \mu (\varphi)$ holds for all
  $\varphi\in C^{\infty}_{\mathsf{c}} (\RR^d)$. Then, by the
  monotone convergence theorem \cite[Thm.~5.5]{Lang}, we have
\[
  \mu (\psi) \, = \lim_{n\to\infty} \mu ( \varphi^{}_{n} \ts \psi)
  \, = \lim_{n\to\infty} T ( \varphi^{}_{n} \ts \psi)
  \, = \, T (\psi) \, < \, \infty \ts ,
\]
  which completes the argument.
\end{proof}

For the general situation, we now get the following result.

\begin{theorem}\label{thm:relations}
  Let\/ $\mu$ be a general Radon measure on\/ $\RR^d$. Then, the
  following properties are equivalent:
\begin{enumerate}\itemsep=2pt
\item $\mu$ is slowly increasing;
\item $\lvert \mu \rvert$ is strongly tempered;
\item $\lvert \mu \rvert$ is tempered;
\item $\lvert \psi \rvert \in L^{1} \bigl( \lvert \mu \rvert \bigr)$
   holds for all\/ $\psi \in \cS (\RR^d )$;
\item $\mu$ is strongly tempered.
\end{enumerate}
Further, if\/ $\mu$ is strongly tempered, it is tempered, while the
converse need not hold.
\end{theorem}

\begin{proof}
  The equivalences of the first four condintions follow immediately
  from Theorem~\ref{thm:positive}, while the implications
\[
   \mu \text{ slowly increasing} \; \Longrightarrow \;
   \mu \text{ strongly tempered} \; \Longrightarrow \;
   \mu \text{ tempered}
\]
  are the result of Lemma~\ref{lem:implications}.

  Finally, the fact that $\mu$ is strongly tempered implies
  $\lvert \psi \rvert \in L^{1} \bigl( \lvert \mu \rvert \bigr)$ for
  all $\psi \in \cS (\RR^d)$ follows directly from the definition.

  \cite[Prop.~7.1]{ARMA1} provides an example of a Radon measure $\mu$
  that is tempered, though $\lvert \mu \rvert$ is not. In particular,
  $\mu$ is not slowly increasing.
\end{proof}

Let us now spell out the implication (4) $\Longrightarrow$ (5) in
Theorem~\ref{thm:relations} more explicitly as follows.

\begin{coro}
  Let\/ $\mu$ be a measure on\/ $\RR^d$, and assume that\/
  $\lvert \psi \rvert \in L^{1} \bigl( \lvert \mu \rvert \bigr)$ holds
  for all\/ $\psi \in \cS (\RR^d )$. Then,
  $\psi\ts \mapsto\nts \int^{}_{\RR^d} \psi(t) \dd \mu(t)$ defines a
  tempered distribution. \qed
\end{coro}

Let us complete the section by another characterisation
when a measure is slowly increasing.

\begin{lemma}\label{lem:linear}
   A Radon measure\/ $\mu$ is slowly increasing if and only if
   it is a linear combination of positive tempered measures.
\end{lemma}

\begin{proof}
  The direction $\Leftarrow$ is obvious. For the converse, consider
  the standard Hahn--Jordan decomposition \cite[Cor.~3.6]{Lang} of
  $\mu = \nu + \ii \sigma$ with $\nu = \real (\mu)$ and
  $\sigma = \imag (\mu)$, that is
\[
    \mu \, = \, \bigl(\nu^{}_{+} \nts - \nu^{}_{-}\bigr)
    + \ii \bigl(\sigma^{}_{+} \nts - \sigma^{}_{-} \bigr) ,
\]
where
$\nu^{}_{\pm} = \frac{1}{2} \bigl( \lvert \nu \rvert \pm \nu \bigr)$
are supported on disjoint sets, and analogously for $\sigma^{}_{\pm}$.
Then, one has
\[
  \lvert \varrho \rvert \, \leqslant \, \lvert \mu \rvert
  \quad \text{for all } \varrho \in \bigl\{
  \nu^{}_{+}, \nu^{}_{-}, \sigma^{}_{+}, \sigma^{}_{-} \bigr\} .
\]
Since $\mu$ is slowly increasing, so are the four components.
\end{proof}

Since the Example in \cite[Prop.~7.1]{ARMA1} is so crucial, but
important details are skipped, let us next construct some examples of
that type more explicitly.

\section{Some tempered measures that are not slowly
increasing}\label{sec:counter}

All the examples in this section fall into the class of tempered
distributions with Fourier transforms in the sense of measures, which
was studied in detail in \cite{ST}. Let us begin with an important
technical step, which we present in all details for clarity and
self-containedness.

\begin{prop}\label{prop:help}
  For every\/ $A>0$, there exists a function\/
  $g\in C_{\mathsf{c}} (\RR)$ with the following three properties:
  $\,\supp (g) \subseteq [-2,2]$, $\| g \|^{}_{1} \geqslant A$, and\/
  $\| \widehat{g} \ts \|^{}_{\infty} \leqslant 1$.
\end{prop}

\begin{proof}
  Define the non-negative functions $f = 1^{}_{[-1,1]}$ and
  $f^{}_{n} = \frac{n}{2} \, 1_{[-\frac{1}{n}, \frac{1}{n}]}$, the
  latter for $n\in\NN$. Clearly, for any $x\in\RR$, one has
\[
  0 \, \leqslant \, \bigl( f * f^{}_{n} \bigr) (x) \, =
  \int_{-1}^{1} f^{}_{n} (x-y) \dd y \, \leqslant
  \int_{\RR} f^{}_{n} (x-y) \dd y \, = \, 1 \ts ,
\]
which implies $\| f * f^{}_{n} \ts \|^{}_{\infty} \leqslant 1$.  Also,
since $\widehat{f^{}_{n}} (t) = \sinc
\bigl(\nts \frac{\ts 2 \pi t \ts}{n}\nts \bigr)$
with $\sinc (z) = \frac{\sin (z)}{z}$, one clearly has pointwise
convergence $\widehat{f^{}_{n}} \xrightarrow{n\to\infty} 1$ on $\RR$.

For any $B>0$, there is an $n\in \NN$ such that
\begin{equation}\label{eq:B-est}
    \big\| \widehat{f^{}_{n} \nts * \nts\nts f} \big\|_{1} \, = \,
    \big\| \widehat{f^{}_{n} } \ts \widehat{f} \, \big\|_{1}
    \, \geqslant \, B \ts .
\end{equation}
To see this, observe that $\widehat{f} (t) = 2 \sinc (2 \pi t )$ is
locally integrable, but gives
$\big\| \widehat{f}\, \big\|_{1} = \infty$ because, for any $N\in\NN$,
one has
\[
  \int_{0}^{\infty} \lvert \sinc (2 \pi t) \rvert \dd t \, \geqslant
  \sum_{n=0}^{N} \int_{\frac{n}{2}}^{\frac{n+1}{2}}
  \frac{\lvert \sin (2 \pi t) \rvert}{2 \pi t} \dd t
  \, \geqslant \sum_{n=0}^{N} \myfrac{1}{\pi (n{+}1)}
  \int_{\frac{n}{2}}^{\frac{n{+}1}{2}} \lvert \sin (2 \pi t) \rvert \dd
  t \, = \, \myfrac{1}{\pi^2} \sum_{n=0}^{N} \myfrac{1}{n{+}1}
 \]
 where the last term is the divergent harmonic series.
 In particular, we have
 \[
      \int_{-\alpha}^{\alpha} \big| \widehat{f} (t) \big|
      \dd t \, > \, 2 B
 \]
for a suitable $\alpha>0$.

Now, recall that
$\widehat{f} (t)\ts \widehat{f^{}_{n}} (t) \xrightarrow{n\to\infty}
\widehat{f} (t)$ holds for any $t\in\RR$, where we also have
$\bigl| \widehat{f^{}_{n}} \bigr| \leqslant 1$. Thus, we have
\[
     \big| 1^{}_{[-\alpha,\alpha]}
     \widehat{f^{}_{n}} \ts \widehat{f} \, \big|
     \,\xrightarrow[n\to\infty]{\,\text{pointwise}\,}\,
     \big| 1^{}_{[-\alpha,\alpha]} \widehat{f} \, \big|
\]
where
$\big| 1^{}_{[-\alpha,\alpha]} \widehat{f^{}_{n}} \ts \widehat{f} \,
\big| $ is dominated by
$1^{}_{[-\alpha,\alpha]} \ts \big| \widehat{f}\, \big| \in L^{1}
(\RR)$. By the dominated convergence theorem
\cite[Thm.~5.8]{Lang}, we thus get
\[
   \lim_{n\to\infty} \int_{-\alpha}^{\alpha}
   \big| \widehat{f^{}_{n}} (t) \ts \widehat{f} (t) \big|
   \dd t \, = \int_{-\alpha}^{\alpha}
   \big| \widehat{f} (t) \big| \dd t \, > \, 2 B \ts.
\]
Consequently, there exists an $n\in\NN$ with
$\int_{-\alpha}^{\alpha} \big| \widehat{f^{}_{n}} (t) \widehat{f} (t)
\big| \dd t > B$, which implies
$\big\| \widehat{f^{}_{n}} \ts \widehat{f} \, \big\|_{1} > B$ and thus
\eqref{eq:B-est}.

Next, let $A>0$ be fixed, and $C>0$ some number that we shall specify
later. If we choose some $B>AC$, there exists an $n\in\NN$ such that
$h = \widehat{f^{}_{n}}\ts \widehat{f}$ satisfies
$\| h \|^{}_{1} \geqslant B > AC$. Then, for a suitable $a > 0$,
we have $\int_{-a}^{a} \lvert h (t) \rvert \dd t > AC$. With
$h^{}_{1} (t) \defeq a \ts h (a t)$, we get
\[
     \int_{-1}^{1}  \lvert h^{}_{1} (t) \rvert \dd t \, =
     \int_{-a}^{a} \lvert h (u) \rvert \dd u
     \, > \, AC \ts ,
\]
together with
$\widehat{h^{}_{1}} (t) = \widehat{h} ( t / a ) = \bigl( f^{}_{n}
* \nts f \bigr) (- t/a)$, which also gives
$\big\| \widehat{h^{}_{1} } \big\|_{\infty} \leqslant 1$.

Fix some $\varphi \in C^{\infty}_{\mathsf{c}} (\RR)$ with
$\varphi \equiv 1$ on $[-1,1]$ and $\supp (\varphi) \subseteq [-2,2]$,
and set $C \defeq \| \widehat{\varphi} \ts\|^{}_{1}$, which clearly
satisfies $C < \infty$. Now, set
$g \defeq C^{-1} \varphi \ts\ts h^{}_{1}$, where
$g \in C_{\mathsf{c}} (\RR)$ is clear. Then,
\[
   \| g \|^{}_{1} \, = \, \myfrac{1}{C} \int_{\RR}
   \lvert \varphi (t) \, h^{}_{1} (t) \rvert \dd t \, \geqslant\,
   \myfrac{1}{C} \int_{-1}^{1} \lvert \varphi (t) \, h^{}_{1} (t)
   \rvert \dd t \, = \, \myfrac{1}{C} \int_{-1}^{1}
   \lvert h^{}_{1} (t) \rvert \dd t \, > \,  A \ts .
\]
This shows $\| g \|^{}_{1} \geqslant A$, and we also have
$\supp (g) \subseteq \supp (\varphi) \subseteq [-2,2]$.

Finally, we have
\[
  \| \widehat{g} \ts \|^{}_{\infty} \, = \, C^{-1} \big\|
  \widehat{\varphi \cdot h^{}_{1}} \big\|_{\infty} \, = \, C^{-1}
  \big\| \widehat{\varphi} *\nts \widehat{h^{}_{1}} \big\|_{\infty} \,
  \leqslant \, C^{-1} \big\| \widehat{\varphi}\ts \big\|_{1} \big\|
  \widehat{h^{}_{1}} \big\|_{\infty} \, = \, \big\| \widehat{h^{}_{1}}
  \big\|_{\infty} \, \leqslant \, 1 \ts ,
\]
which proves the claim.
\end{proof}

\begin{remark}
  It is important to note that the function $g$, once $A$ is large
  enough, cannot be a positive function. If it were, we would get
  $\widehat{g} (0) = \| g \|^{}_{1} > A$, in contradiction to
  $\| \widehat{g} \|^{}_{\infty} \leqslant 1$. The analogous comment
  also applies to the function $h^{}_{1}$ constructed in the proof.
  \exend
\end{remark}

This has the following consequence, which is also part of
\cite[Prop.~7.1]{ARMA1}.

\begin{coro}\label{coro:gn}
  There is a sequence\/ $(g^{}_{n})^{}_{n\in\NN}$ of functions\/
  $g^{}_{n} \in C_{\mathsf{c}} (\RR)$ with the following three
  properties:
  $\, \supp (g^{}_{n}) \subseteq \bigl[ -\frac{1}{n+1}, \frac{1}{n+1}
  \bigr]$, $\| g^{}_{n} \|^{}_{1} \geqslant (n^2+1)^{n}$, and\/
  $\| \widehat{\ts g^{}_{n}} \|^{}_{\infty} \leqslant 2^{-n}$.
\end{coro}

\begin{proof}
  Let $A>0$ and $g$ be as in Proposition~\ref{prop:help}, and let
  $\beta,\gamma > 0$ be arbitrary. If we set
  $h^{}_{\beta,\gamma} (x) \defeq \frac{1}{\beta} \, g (\gamma x)$, we
  get
  $\supp (h^{}_{\beta,\gamma} ) \subseteq \bigl[ -\frac{2}{\gamma} ,
  \frac{2}{\gamma} \bigr]$ together with
\[
   \big\| h^{}_{\beta,\gamma} \big\|_{1} \, = \, \beta^{-1} \!\int_{\RR}
   \, \lvert g (\gamma x )\rvert \dd x \, = \, \frac{\| g \|^{}_{1} }
   {\beta\gamma} \, \geqslant \, \myfrac{A}{\beta\gamma}
\]
and
$\widehat{h^{}_{\beta,\gamma}} (t) = \frac{1}{\beta\gamma} \,
\widehat{g} ( t/\gamma ) $, hence
$\big\| \widehat{h^{}_{\beta,\gamma}} \big\|_{\infty} \leqslant
\frac{1}{\beta\gamma}$.  Choosing $\gamma = 2 (n+1)$ and
$\beta = 2^{n-1}/(n+1)$, the claim follows with $g^{}_{n} =
h^{}_{\beta,\gamma}$ and $A = 2^n (n^2 + 1)^{n}$.
\end{proof}

Let us next prove a result that is the key to go beyond the example of
\cite[Prop.~7.1]{ARMA1}. First, let us recall that a sequence
$(\mu^{}_{n})^{}_{n\in\NN}$ of measures is said to \emph{converge
  vaguely} to a measure $\mu$ if, for all $\varphi \in \Cc(\RR^d)$, we
have
$\mu^{}_{n}(\varphi) \xrightarrow{\, n\to\infty \,} \mu(\varphi)$.

\begin{lemma}\label{lem:temp-not-slw}
  Let\/ $A>0$ and let\/ $(\nu^{}_{n})^{}_{n\in\NN}$ be a sequence of
  finite measures on\/ $\RR^d$ with the following three properties:
  $\, \supp (\nu^{}_{n}) \subseteq [-A, A ]^d$,
  $\; \lvert \nu^{}_{n} \rvert (\RR^d) \geqslant (n^2+1)^{n}$, and\/
  $\, \| \widehat{\ts \nu^{}_{n}} \|^{}_{\infty} \leqslant 2^{-n}$.
  Set\/ $v = ( 4 \ts A , 0, \ldots, 0)$ and define\/
\[
   \mu^{}_n \, =
   \sum_{m=1}^n \delta^{}_{m v} \nts* \nu^{}_m \ts .
\]
Then, $(\mu^{}_{n})^{}_{n\in\NN}$ converges vaguely to a measure\/
$\mu$ that is tempered but not slowly increasing.
\end{lemma}

\begin{proof}
  Let $C_{\mathsf{u}} (\RR^d)$ denote the space of bounded, uniformly
  continuous functions on $\RR^d$ and consider
  $H^{}_{n}\nts = \widehat{\mu^{}_{n}}$ as defined by
\[
   \widehat{\mu^{}_{n}} (x) \, =
  \sum_{m=1}^{n} \ee^{-2 \pi \ii \ts m \ts v \cdot x } \,
  \widehat{\nu^{}_{m}} (x) \ts ,
\]
where ${v \cdot x} = 4\ts A \ts x^{}_{1}$,
which clearly satisfies $H_n \in C_{\mathsf{u}} (\RR^d)$. As
$\| \widehat{\nu^{}_{n}} \|^{}_{\infty} < 2^{-n}$, the sequence
$ ( H^{}_n )^{}_{n\in\NN}$ converges, in
$\bigl( C_{\mathsf{u}} (\RR^d), \| . \|^{}_{\infty} \bigr)$, to some
$H \in C_{\mathsf{u}} (\RR^d)$.

To continue, it follows immediately from the definition of
$\mu^{}_{n}$ that, for any $\varphi \in C_{\mathsf{c}} (\RR^d)$, there
is an integer $N=N(\varphi)$ such that
$\mu^{}_{N} (\varphi) = \mu^{}_{N+k} (\varphi)$ holds for all
$k\geqslant 1$. Consequently, $(\mu^{}_{n})^{}_{n\in\NN}$ is vaguely
Cauchy, and hence vaguely convergent so some Radon measure $\mu$ on
$\RR^d$. Moreover, for any $\varphi \in C_{\mathsf{c}} (\RR^d)$, one
has
\[
  \mu (\varphi) \, = \, \mu^{}_{n} (\varphi)
  \qquad \text{for all } n \geqslant N (\varphi) \ts ,
\]
with the $N (\varphi)$ from above. Next,
\[
  T_n (\psi) \, \defeq \int_{\RR^d} \psi (-t) \,
  H_{n} (t) \dd t  \quad \text{and} \quad
  T (\psi) \, \defeq \int_{\RR^d} \psi (-t) \, H (t) \dd t
\]
define tempered distributions, with $T_n \xrightarrow{n\to\infty} T$
in $\cS' (\RR)$.

Further, by \cite[Lemma~4.9.14]{MoSt} and \cite[Prop.~3.1]{ARMA1}, we
have
\[
  \mu^{}_n(\varphi)\, = \int_{\RR^d} \widecheck{\varphi}(t) \dd
  \widehat{\mu^{}_n}(t) \, = \, T^{}_n (\widehat{\varphi}\ts )
\]
for all $\varphi \in \Cc^{\infty}(\RR^d)$.  This gives
\[
  \widehat{T} (\varphi) \, = \, T (\widehat{\varphi} \ts )
  \, = \lim_{n\to\infty} T_n (\widehat{\varphi}\ts ) \, =
  \lim_{n\to\infty} \mu_{n} (\varphi) \, = \, \mu (\varphi) \ts ,
\]
which shows that $\mu$ is indeed tempered.

Finally, let $n\geqslant 3$ and let
  $\varphi \in C_{\mathsf{c}} (\RR^d)$ satisfy
  $\supp (\varphi) \subseteq n v + \bigl[ -A, A
  \bigr]^d$. Then, due to our construction, we have
\[
  \mu^{}_{m} (\varphi) \, = \int_{\RR^d} \varphi (x) \, \dd
  \bigl( \delta^{}_{n v} \! * \nu^{}_n \bigr) (x)
  \qquad \text{for all } m>n \ts ,
\]
and thus also
$\mu (\varphi) = \int_{\RR^d} \varphi (x) \dd \bigl( \delta^{}_{nv}
 \! * \nu^{}_n \bigr) (x)$. This shows that
\[
  \mu\big|_{n v + [ -A, A ]^d} \, = \,
  \delta^{}_{ n v} \! * \nu^{}_n \ts .
\]
In particular, for any fixed $P \in \RR [x^{}_1,\ldots, x^{}_d]$, we
get
\[
  \int_{\RR^d} \frac{\dd\ts \lvert \mu \rvert (x)} {1 + \lvert P(x)\rvert}
  \, \geqslant \int^{}_{\left[ -A, A\right]^d}
  \frac{ \dd \ts\lvert \nu^{}_n  \rvert (x^{}_1,\ldots, x^{}_d)}
  {1 + \lvert  P(x^{}_1 -4 A n, x^{}_2, \ldots, x^{}_d)
    \rvert} \, \xrightarrow{\, n\to\infty\,} \, \infty \ts ,
\]
where the last claim follows immediately from
$\left| \nu^{}_{n} \right| \bigl( [-A, A ]^d \ts \bigr) \geqslant
(n^2+1)^{n}$.  This shows that $\mu$ cannot be slowly increasing.
\end{proof}

\begin{remark}\label{rem:supp-mu}
  By construction, the measures $\mu^{}_n$ and $ \mu$ from
  Lemma~\ref{lem:temp-not-slw} satisfy the following simple
  property. For each compact set $K \subset \RR^d$, there
  exists some integer $N=N(K)$ such that, for all $n >N$, we
  have $ \mu |^{}_{K}  =     \mu^{}_{n} |^{}_{K}$, where
  $\mu |^{}_{K}$ denotes the restriction of $\mu$ to $K$.
  \exend
\end{remark}

Now, setting $\nu^{}_n = g^{}_n \lambda^{}_{\mathrm{L}}$ with $g^{}_n$
as in Corollary~\ref{coro:gn} and $\lambda^{}_{\mathrm{L}}$ denoting
Lebesgue measure, Proposition~\ref{lem:temp-not-slw} gives the
following concrete version of \cite[Prop.~7.1]{ARMA1}.

\begin{prop}\label{prop:no-way}
  Let\/ $g^{}_{n}$ be as in Corollary~\textnormal{\ref{coro:gn}} and
  consider the measures defined by
\[
  \mu^{}_n (\varphi)\, \defeq \sum_{j=1}^n
  \int^{}_{\RR} \varphi(x) \, g^{}_j(x+j) \, \dd x \ts .
\]
Then, the sequence\/ $(\mu^{}_n )^{}_{n\in\NN}$ converges vaguely to a
signed Radon measure\/ $\mu$ that is tempered but not slowly
increasing.  \qed
\end{prop}

Later, in Example~\ref{ex:final}, we shall provide an example of a
tempered measure with locally finite support that is not slowly
increasing. Before we can do this, we need to discuss the case
of measures with uniformly discrete support more generally.

\section{Radon measures with uniformly discrete
support}\label{sec:finite}

Here, we consider the important special case of measures with
uniformly discrete support, for which the three key notions turn out
to be equivalent. This class is particularly relevant in the theory of
aperiodic order, with several applications to mathematical
quasicrystals and Meyer sets; see
\cite{BST,Jeff-rev,LO1,Meyer,RS,Nicu,NS11} and references therein.

Note first that, if $\mu$ is tempered, strongly tempered, or slowly
increasing, the same property holds for $\overline{\mu}$.  This has
the following immediate consequence.

\begin{fact}\label{fact:two}
   If\/ $\mu$ is a Radon measure on\/ $\RR^d$, one has
\begin{enumerate}\itemsep=2pt
\item  $\mu$ is tempered $\, \Longleftrightarrow\,$
   $\real (\mu)$ and\/ $\imag (\mu)$ are tempered{\ts};
\item $\mu$ is slowly increasing $\, \Longleftrightarrow\,$
   $\real(\mu)$ and\/ $\imag (\mu)$ are slowly increasing. \qed
\end{enumerate}
\end{fact}

To continue, we need a simple separation result as follows.

\begin{lemma}\label{lem:help}
  Let\/ $U,V \subset \RR^d$ be such that\/ $U\cup V$ is uniformly
  discrete and\/ $U\cap V = \varnothing$. Then, there exists a
  function\/ $f \in C^{\infty} (\RR^d)$ such that\/ $f$ and all its
  derivatives are bounded, together with\/ $f(x) = 1$ for all\/
  $x\in U$ and\/ $f(y)=0$ for all\/ $y\in V$.
\end{lemma}

\begin{proof}
  Let $r>0$ be such that, for all $x, y \in U \cup V$ with $x \neq y$,
  we have $B_r(x) \cap B_r(y) = \varnothing$. Let
  $\varphi \in \Cc^{\infty}(\RR^d)$ be so that $\varphi(0) = 1$ together
  with $\supp (\varphi) \subseteq B^{}_r (0)$ and
  $\varphi(x) \in [\ts 0,1]$ for all $x$, which exists by standard
  arguments.

  Define $f = \sum_{u\in U} T^{}_u \varphi$, where
  $\bigl(T^{}_u \varphi \bigr) (x) = \varphi (x-u)$. Since
  $\supp(T^{}_u \varphi) \subseteq B_{r}(u)$, where the $B_{r}(u)$ are
  pairwise disjoint open sets, it is immediate that $f$ has the
  desired properties.
\end{proof}

Next, let us recall the following standard result, which we prove for
convenience.

\begin{fact}\label{fact:f-bound-fT-temp}
  Let\/ $f \in C^{\infty}(\RR^d)$ and\/ $T \in \cS'(\RR^d)$. If\/ $f$
  and all its derivatives are bounded, the mapping\/
  $\varphi \mapsto \bigl( fT \bigr)(\varphi) \defeq T(f \varphi)$
  defines a tempered distribution.
\end{fact}

\begin{proof}
Define $F \colon \cS(\RR^d) \xrightarrow{\quad} \Cu(\RR^d)$ via
\[
   F(\varphi) \, \defeq \, f \varphi \,.
\]
Let $\alpha, \beta$ be arbitrary multi-indices, with ordering
defined componentwise, and set
\[
  \binom{\alpha}{\beta} \, = \, \prod_{i=1}^{d}
  \binom{\alpha^{}_i}{\beta^{}_i} \ts .
\]
Then, via the multivariate derivation formula of Leibniz, we have
\begin{align*}
    \big| x^{\beta} D^{\alpha}( f \varphi) \big| \, & = \,
    \biggl| \ts\ts x^{\beta}  \sum_{ \gamma \leqslant \alpha}
    \binom{\alpha}{\gamma} \, ( D^{\gamma} f ) \,
    ( D^{\alpha-\gamma}\varphi) \, \biggr|
    \, \leqslant  \ts \sum_{ \gamma \leqslant \alpha}
    \binom{\alpha}{\gamma}  \big| ( D^{\gamma} f ) \,
     x^{\beta} (D^{\alpha-\gamma}\varphi)\ts \big|  \\
  & \leqslant  \ts \sum_{ \gamma \leqslant \alpha}
  \binom{\alpha}{\gamma} \|  D^{\gamma} f
  \|^{}_\infty \,  \|\varphi \|^{}_{\beta, \alpha -\gamma}
\end{align*}
for every $x\in\RR^d$.  This shows that
$F\bigl( \cS(\RR^d) \bigr) \subseteq \cS(\RR^d)$ and that
$F\colon \cS(\RR^d) \xrightarrow{\quad} \cS(\RR^d)$ is continuous with
respect to the Schwartz topology. In particular,
$f\ts T =T {\ts\circ\ts} F \in \cS'(\RR^d)$.
\end{proof}

The equivalence of the key notions in this case can now be stated
as follows.

\begin{theorem}\label{thm:loc-fin-supp}
  Let\/ $\mu$ be a Radon measure on\/ $\RR^d$ with uniformly discrete
  support. If\/ $\mu$ is tempered, it is also slowly increasing.
\end{theorem}

\begin{proof}
  Since $\mu$ is a tempered measure with uniformly discrete support,
  so are $\mbox{Re}(\mu)$ and $\mbox{Im}(\mu)$. If we show the latter
  to be slowly increasing, $\mu$ is slowly increasing by
  Facy~\ref{fact:two}. Thus, without loss of generality, we may assume
  $\mu$ to be a signed measure. Define
\[
    \vL^{}_{\pm} \, = \, \{ x \in \RR^d : \mu (\{ x \} )
    \,\gtrless \, 0 \} \ts .
\]
Then, the set $\vL = \{ x \in \RR^d : \mu (\{ x \} ) \ne 0 \}$, which
is the support of $\mu$ and uniformly discrete by assumption,
satisfies $\vL = \vL_+ \cup \vL_-$ together with
$\vL_+ \nts \cap \vL_- = \varnothing$.

If $\mu$ is tempered, there is a $T\in\cS' (\RR^d)$ such that
$\mu (\varphi) = T (\varphi)$ holds for all
$\varphi \in C^{\infty}_{\mathsf{c}} (\RR^d)$.  Let
$f \in C^{\infty} (\RR^d)$ be a function such that $f$ and all its
derivatives are bounded with $f |^{}_{\vL_{+}} \equiv 1$ and
$f |^{}_{\vL_{-}} \equiv 0$, which is guaranteed to exist by
Lemma~\ref{lem:help}, and set $g = 1-f$, so also
$g\in C^{\infty} (\RR^d)$ and $g$ and all its derivatives are bounded.

Setting $\mu^{}_{+} = f \cdot \mu$ and $\mu^{}_{-} = (-g)\cdot \mu$,
we get $\mu = \mu^{}_{+} \nts - \mu^{}_{-}$ where $\mu^{}_{+}$ and
$\mu^{}_{-}$ are positive Radon measures by construction. Further, for
all $\varphi\in C^{\infty}_{\mathsf{c}} (\RR^d)$, we have
\[
    \mu^{}_{+} (\varphi) \, = \, \bigl( f \mu \bigr) (\varphi) \, = \,
    \mu (f \varphi ) \, = \, T (f \varphi) \, = \, \bigl( f \ts T \bigr)
    (\varphi) \ts .
\]
Since $T\in \cS' (\RR^d)$ with $f\in C^{\infty} (\RR^d)$ and $f$ and
all its derivatives are bounded, we have $f\ts T \in \cS' (\RR^d)$ by
Fact~\ref{fact:f-bound-fT-temp}.  Therefore, $\mu^{}_{+}$ is a
positive, tempered measure, hence also slowly increasing.

In the same way, one gets
$\mu^{}_{-} (\varphi) = (g \ts T) (\varphi)$, hence $\mu^{}_{-}$ is
slowly increasing as well.
\end{proof}

Explicitly, we can summarise the situation as follows.

\begin{coro}
  Let\/ $\mu$ be a Radon measure on\/ $\RR^d$ with uniformly discrete
  support. Then, the following properties are equivalent:
\begin{enumerate}\itemsep=2pt
\item $\mu$ is slowly increasing;
\item $\mu$ is strongly tempered;
\item one has\/ $\lvert \psi \rvert \in L^{1} (\mu)$
  for all\/ $\psi \in \cS (\RR^d)$;
\item one has\/ $\psi \in L^{1} \bigl( \lvert \mu \rvert \bigr)$ for
  all\/ $\psi \in \cS (\RR^d)$ with\/ $\psi \geqslant 0$;
\item $\mu$ is tempered. \qed
\end{enumerate}
\end{coro}

\begin{remark}
  If $U \cup V$ is locally finite, looking at the the proof of
  Lemma~\ref{lem:help}, we can still select radii $r_u>0$ for the
  points $u\in U$ such that $B_{r_u}(u) \cap (U \cup V)= \{ u\}$.
  Further, we can find functions
  $\varphi^{}_{u} \in C^{\infty}_{\mathsf{c}} (\RR^d)$ so that
  $\varphi^{}_{u} (u) = 1$ together with
  $\supp (\varphi^{}_{u}) \subseteq B^{}_r (u)$ and
  $\varphi^{}_{u} (x) \in [\ts 0,1]$ for all $x$. Then, via
  $f = \sum_{u\in U} \varphi^{}_{u}$, we get a function
  $f \in C^{\infty}(\RR^d)$ that is bounded.

  However, if $U \cup V$ is locally finite but \emph{not} uniformly
  discrete, the radii $r_u$ get arbitrarily close to zero. This forces
  the derivatives of $f$ to become unbounded. Consequently, in the
  proof of Theorem~\ref{thm:loc-fin-supp}, $fT$ is a distribution that
  need no longer be tempered. This shows that our proof of
  Theorem~\ref{thm:loc-fin-supp} cannot be extended to general
  measures with locally finite support.  In fact, we shall see in the
  next example that there exist tempered pure point measures with
  locally finite support that are not slowly increasing.  \exend
\end{remark}

Employing the construction of \cite{KS}, we now show that
Theorem~\ref{thm:loc-fin-supp} does not hold for measures with locally
finite support.

\begin{example}\label{ex:final}
  For distinct, positive numbers $a,b \in \RR$, consider
  $\mu^{}_{a,b} \defeq \delta^{}_0 +\delta^{}_a+\delta^{}_b -
  \delta^{}_{a+b}\ts $. Then, as observed in \cite{KS}, we have
  $\| \mu^{}_{a,b} \| =4$ and
\begin{equation}\label{eq:KS-ineq}
  \| \widehat{\mu^{}_{a,b}} \| \, \leqslant \, 2 \sqrt{2} \ts ,
\end{equation}
because a simple calculation with $z=\ee^{-2\pi\ii \ts a \cdot x}$ and
$w=\ee^{-2\pi\ii \ts b \cdot x}$ shows that
\[
    \| \widehat{\mu^{}_{a,b}} \|^2 \, = \,
    (1+z+w-zw)(1+\overline{z}+\overline{w}-\overline{zw})
    \, = \, 4-\overline{zw}+z\overline{w}+w\overline{z} -zw \ts ,
\]
which is a positive number, so we get
$\| \widehat{\mu^{}_{a,b}} \|^2 \leqslant 8$ via the triangle
inequality.

Next, for each $n$, select numbers
$a^{}_{1}, \ldots, a^{}_{n}, b^{}_1, \ldots, b^{}_n \in (0,1]$ that
are linearly independent over $\QQ$. Then, the elements
$k^{}_1 a^{}_1 + \ldots + k^{}_n a^{}_n +\ell^{}_1 b^{}_1 + \ldots +
\ell^{}_n b^{}_n$ are distinct for all $2^{2n}$ choices of
$k^{}_1, \ldots, k^{}_n$ and $\ell^{}_1 , \ldots, \ell^{}_n $ in
$ \{ 0,1 \}$.

Now, consider
\[
   \nu^{}_n \, \defeq \, \Conv_{i=1}^n \mu^{}_{a^{}_i , b^{}_i} \ts .
\]
A simple computation shows that $\nu^{}_{n}$ has the form
\begin{equation}\label{eq:KS-example}
  \nu^{}_n \, =  \sum_{k^{}_1 , \ldots, k^{}_n , \ell^{}_1 ,
    \ldots, \ell^{}_n \in \{ 0,1 \}}
  s \ts ( k^{}_1 , \ldots, k^{}_n , \ell^{}_1 , \ldots, \ell^{}_n )
  \, \delta^{}_{ k_1 a_1+ \ldots + k_n a_n + \ell_1 b_1+ \ldots + \ell_n b_n }
\end{equation}
with
\[
  s \ts ( k^{}_1 , \ldots, k^{}_n , \ell^{}_1 ,
  \ldots, \ell^{}_n ) \, = \, (-1)^{ \card \{ i :
    1 \leqslant i \leqslant n , k_i = \ell_i = 1 \} }
  \, = \, \pm 1 \ts .
\]
Since the Dirac measures on the RHS of \eqref{eq:KS-example} have
pairwise disjoint supports, we get
\[
  \| \nu^{}_n \| \, =  \sum_{k^{}_1 , \ldots, k^{}_n ,
    \ell^{}_1 , \ldots, \ell^{}_n \in \{ 0,1 \}}
  \lvert s \ts ( k^{}_1 , \ldots, k^{}_n , \ell^{}_1 ,
  \ldots, \ell^{}_n ) \rvert \, = \, 2^{2n} .
\]
Moreover, Eq.~\eqref{eq:KS-ineq} implies
\[
  \| \widehat{\nu^{}_{n}} \|^{}_\infty \, = \,
  \Bigl\| \prod_{i=1}^{n} \widehat{\mu^{}_{a^{}_i , b^{}_i }}
  \Bigr\|_\infty \, \leqslant \, 2^{\frac{3 n}{2}}_{\vphantom{I}} .
\]

Now, for each $m \in \NN$, pick some $n$ such that
$2^{\frac{n}{2}}_{\vphantom{I}} \geqslant 2^m(m^2+1)^m$, and consider
\[
  \omega^{}_{m} \, \defeq \, \frac{\nu^{}_{n}}
  {2^m_{\vphantom{I}} \ts \| \widehat{\nu_n}\|^{}_{\infty}}  \ts .
\]
Then, we get
\[
  \|\ts \omega^{}_{m} \| \, = \,
  \frac{\| \nu^{}_{n} \|}{2^{m}_{\vphantom{I}} \|
    \ts \widehat{\nu_n}\|^{}_{\infty}}
  \, \geqslant \, 2^{\frac{n}{2} - m}_{\vphantom{I}}
  \, \geqslant \, (m^2+1)^{m}
  \qquad \text{and} \qquad
  \| \ts \widehat{\omega^{}_{m}} \|^{}_\infty \, = \, 2^{-m} .
\]
Further, by construction,
$\supp(\omega^{}_m) \subseteq [0,2] \subseteq [-2,2]$. Therefore, by
Lemma~\ref{lem:temp-not-slw}, the measure
\[
   \mu \, =
   \sum_{m=1}^{\infty} \delta^{}_{8 m } \nts* \omega^{}_m
\]
is tempered, but not slowly increasing. Moreover, since each
$\omega^{}_m$ has finite support, $\mu$ has locally finite support by
Remark~\ref{rem:supp-mu}.
\end{example}

\section*{Acknowledgements}

We are grateful to Timo Spindeler for valuable discussions and helpful
comments on the manuscript. We thank an anonymous referee for several
helpful comments.  This work was supported by the German Research
Foundation (DFG, Deutsche Forschungsgemeinschaft), within the CRC
1283/2 \mbox{(2021 - 317210226)} at Bielefeld University (MB) and by
the Natural Sciences and Engineering Council of Canada (NSERC), via
grant 2020-00038 (NS).

\end{document}